\documentclass[preprint]{elsarticle}
\usepackage{array}
\usepackage{amsmath, amsthm, amssymb,amsfonts}			% Std packages
\usepackage{mathrsfs}									% Script
\usepackage{multirow}
\usepackage{hyperref}
\usepackage{float}
\usepackage{kbordermatrix}								% Fancy matrices
\usepackage{changepage}
\usepackage[usenames,dvipsnames,svgnames,table]{xcolor}
\usepackage{enumitem}

\allowdisplaybreaks

\usepackage{pgfplots}
\usepgfplotslibrary{external}
\newtheorem{thm}{Theorem}[section]
\newtheorem{lem}[thm]{Lemma}
\newtheorem{cor}[thm]{Corollary}
\newtheorem{prop}[thm]{Proposition}

\theoremstyle{definition}			                % Subseq. Thm Style
\newtheorem{mydef}[thm]{Definition}
\newtheorem{rem}[thm]{Remark}
\newtheorem{ex}[thm]{Example}
\newtheorem{prob}[thm]{Problem}

\numberwithin{equation}{section}		% Number eqns within section

\newcommand{\bb}[1]{\mathbb{#1}}			% Blackboard boldface for R, C, N, Q, etc.
\newcommand{\ii}{\textup{i}}									% \textup{i} := \sqrt{-1}
\newcommand{\mat}[2]{\textsl{M}_{#1}(#2)}						% nxn real matrices
\newcommand{\gl}[2]{GL_{#1}(#2)}								% General Linear Group
\newcommand{\og}[1]{O\left(#1\right)}							% Orth Group
\newcommand{\bracket}[1]{\langle #1 \rangle}						% <n>
\newcommand{\inv}[1]{#1^{-1}}								% inverse

\newcommand{\trace}[1]{\operatorname{\rm tr}\left( #1 \right)}				% matrix trace
\newcommand{\vz}[1]{\operatorname{\rm vec}\left( #1 \right)}				% vectorization
\newcommand{\diag}[1]{\operatorname{\rm diag}\left( #1 \right)}			% diagonal matrix
\newcommand{\Diag}[1]{\operatorname{\rm Diag}\left( #1 \right)}			% diagonal of square matrix
\newcommand{\rk}[1]{\operatorname{\rm rank}\left(#1\right)}				% matrix rank
\newcommand{\spn}[1]{\operatorname{\rm span}\left(#1\right)}				% span
\newcommand{\conv}[1]{\operatorname{\rm conv}\left( #1 \right)}			% Convex Hull
\newcommand{\coni}[1]{\operatorname{\rm coni}\left( #1 \right)}			% Conical Hull
\newcommand{\interior}[1]{\operatorname{int}\left( #1 \right)}				% Interior
\newcommand{\vol}[1]{\operatorname{Vol}\left( #1 \right)}				% Volume

\newcommand{\sr}[1]{\rho\left(#1\right)}							% spectral radius
\newcommand{\sig}[1]{\sigma \left( #1 \right)}						% multi-set of eigenvalues
\newcommand{\cone}[1]{\mathcal{C}\left( #1 \right)}					% Spectracone
\newcommand{\tope}[1]{\mathcal{P}\left( #1 \right)}					% Spectratope

\newcommand{\hyp}[2]{#1 \hyperref[#2]{\ref*{#2}}}					% Hyperlink
\journal{Linear Algebra and its Applications}

\begin{document}
% Front matter--------------------------------------------------------------
\begin{frontmatter}
\title{Perron Spectratopes and the Real Nonnegative Inverse Eigenvalue Problem}

\author[addy1]{Charles R.~Johnson}										
\ead{crjohn@wm.edu}

\author[addy2]{Pietro Paparella\corref{corpp}}
\ead{pietrop@uw.edu}
\ead[url]{http://faculty.washington.edu/pietrop/}

\cortext[corpp]{Corresponding author.}

\address[addy1]{Department of Mathematics, College of William \& Mary, Williamsburg, VA 23187-8795, USA}
\address[addy2]{Division of Engineering and Mathematics, University of Washington Bothell, Bothell, WA 98011-8246, USA }

\begin{abstract}
Call an $n$-by-$n$ invertible matrix $S$ a \emph{Perron similarity} if there is a real non-scalar diagonal matrix $D$ such that $S D \inv{S}$ is entrywise nonnegative. We give two characterizations of Perron similarities and study the polyhedra $\cone{S} := \{ x \in \bb{R}^n: S D_x \inv{S} \geq 0,~D_x := \diag{x} \}$ and $\tope{S} := \{x \in \cone{S} : x_1 = 1 \}$, which we call the \emph{Perron spectracone} and \emph{Perron spectratope}, respectively. The set of all normalized real spectra of diagonalizable nonnegative matrices may be covered by Perron spectratopes, so that enumerating them is of interest. 

The Perron spectracone and spectratope of Hadamard matrices are of particular interest and tend to have large volume. For the canonical Hadamard matrix (as well as other matrices), the Perron spectratope coincides with the convex hull of its rows. 

In addition, we provide a constructive version of a result due to Fiedler (\cite[Theorem 2.4]{f1974}) for Hadamard orders, and a constructive version of \cite[Theorem 5.1]{bh1991} for Sule\u{\i}manova spectra.
\end{abstract}

\end{frontmatter}

%-------------------------------------------------------------------------------------------------------------------------------------------------------------------------------------
\section{Introduction}
%-------------------------------------------------------------------------------------------------------------------------------------------------------------------------------------
 
The \emph{real nonnegative inverse eigenvalue problem} (RNIEP) is to determine which sets of $n$ real numbers occur as the spectrum of an $n$-by-$n$ nonnegative matrix. The RNIEP is unsolved for $n \geq 5$, and the following variations, which are also unsolved for $n \geq 5$, are relevant to this work (additional background information on the RNIEP can be found in, e.g., \cite{eln2004}, \cite{mps2007}, and \cite{m1988}):
\begin{itemize} 
\item \emph{Diagonalizable RNIEP} (D-RNIEP): Determine which sets of $n$ real numbers occur as the spectrum of an $n$-by-$n$ diagonalizable nonnegative matrix.

\item \emph{Symmetric NIEP} (SNIEP): Determine which sets of $n$ real numbers occur as the spectrum of an $n$-by-$n$ symmetric nonnegative matrix.

\item \emph{Doubly stochastic RNIEP} (DS-RNIEP): Determine which sets of $n$ real numbers occur as the spectrum of an $n$-by-$n$ doubly stochastic matrix.

\item \emph{Doubly stochastic SNIEP} (DS-SNIEP): Determine which sets of $n$ real numbers occur as the spectrum of an $n$-by-$n$ symmetric doubly stochastic matrix.
\end{itemize}
The RNIEP and the SNIEP are equivalent when $n \leq 4$ and distinct otherwise (see \cite{jll1996}). Notice that there is no distinction between the the SNIEP and the D-SNIEP since every symmetric matrix is diagonalizable.

The set $\sigma = \{ \lambda_1, \dots, \lambda_n \} \subset \bb{R}$ is said to be \emph{realizable} if there is an $n$-by-$n$ nonnegative matrix with spectrum $\sigma$. If $A$ is a nonnegative matrix that realizes $\sigma$, then $A$ is called a \emph{realizing matrix} for $\sigma$. It is well-known that if $\sigma$ is realizable, then  
\begin{align}
%\item 
s_k (\sigma) &:= \sum_{i=1}^n \lambda_i^k \geq 0,~\forall~k \in \bb{N} 	\label{trnn}	\\ 
%\item 
\sr{\sigma} &:= \max_i |\lambda_i| \in \sigma						\label{sprad}	\\ 
%\item 
s_k^m (\sigma) &\leq n^{m-1} s_{km}, \forall~k, m \in \bb{N}  			\label{JLL}  
\end{align}
Condition \eqref{JLL}, known as the \emph{J-LL condition}, was proven independently by Johnson in \cite{j1981}, and by Loewy and London in \cite{ll1978-79}. 

In this paper, we introduce several polyhedral sets whose points correspond to spectra of entrywise nonegative matrices. In particular, given a nonsingular matrix $S$, we define several polytopic subsets of the polyhedral cone $\cone{S} := \{ x \in \bb{R}^n: S D_x \inv{S} \geq 0,~D_x := \diag{x} \}$ and use them to verify the known necessary and sufficient conditions for the RNIEP and SNIEP when $n \leq 4$. For a nonsingular matrix $S$, we provide a necessary and sufficient condition such that $\cone{S}$ is nontrivial. For every $n \geq 1$, we characterize $\mathcal{C}(H_n)$, where $H_n$ is the \emph{Walsh matrix} of order $2^n$, which resolves a problem posed in \cite[p.~48]{e2009}. Our proof method yields a highly-structured $(2^n-1)$-class \emph{(commutative) association scheme} and, as a consequence, a highly structured \emph{Bose-Mesner Algebra}. In addition, we provide a constructive version of a result due to Fiedler (\cite[Theorem 2.4]{f1974}) for Hadamard orders, and a constructive version of \cite[Theorem 5.1]{bh1991} for Sule\u{\i}manova spectra. The introduction of these convex sets extends techniques and ideas found in (e.g.) \cite{enn1998, e2009, em2009, p1952,p1953,p1955,s1983} and provides a framework for investigating the aforementioned problems.     

%-------------------------------------------------------------------------------------------------------------------------------------------------------------------------------------
\section{Notation and Background}
%-------------------------------------------------------------------------------------------------------------------------------------------------------------------------------------

Denote by $\bb{N}$ the set of natural numbers and by $\bb{N}_0$ the set $\bb{N} \cup \{ 0 \}$. For $n \in \bb{N}$, the set $\{ 1,\dots, n\} \subset \bb{N}$ is denoted by $\bracket{n}$. If $\sigma = \{ \lambda_1, \dots, \lambda_n \} \subset \bb{R}$, then $\sigma$ is called \emph{normalized} if $\lambda_1 = 1 \geq \lambda_2 \geq \dots \geq \lambda_n$. 

The set of $m$-by-$n$ matrices with entries from a field $\bb{F}$ (in this paper, $\bb{F}$ is either $\bb{C}$ or $\bb{R}$) is denoted by $\mat{m,n}{\bb{F}}$ (when $m = n$, $\mat{n,n}{\bb{F}}$ is abbreviated to $\mat{n}{\bb{F}}$). The set of all $n$-by-$1$ column vectors is identified with the set of all ordered $n$-tuples with entries in $\bb{F}$ and thus denoted by $\bb{F}^m$. The set of nonsingular matrices over $\bb{F}$ is denoted by $\gl{n}{\bb{F}}$ and the set of $n$-by-$n$ orthogonal matrices is denoted by $\og{n}$. 

For $A = [a_{ij}] \in \mat{n}{\bb{C}}$, the \emph{transpose} of $A$ is denoted by $A^\top$; the \emph{spectrum} of $A$ is denoted by $\sigma = \sig{A}$; the \emph{spectral radius} of $A$ is denoted by $\rho = \sr{A}$; and $\Diag{A}$ denotes the $n$-by-$1$ column vector $[a_{ii}~\cdots~a_{nn}]^\top$. Given $x \in \bb{F}^n$, $x_i$ denotes the $i\textsuperscript{th}$ entry of $x$ and $D_x = D_{x^\top} \in \mat{n}{\bb{F}}$ denotes the diagonal matrix whose $(i,i)$-entry is $x_i$. Notice that for scalars $\alpha$, $\beta \in \bb{F}$, and vectors $x$, $y \in \bb{F}^n$, $D_{\alpha x + \beta y} = \alpha D_x + \beta D_y$.

The \emph{Hadamard product} of $A$, $B \in \mat{m,n}{\bb{F}}$, denoted by $A \circ B$, is the $m$-by-$n$ matrix whose $(i,j)$-entry is $a_{ij} b_{ij}$. The \emph{direct sum} of $A_1, \dots, A_k$, where $A_i \in \mat{n_i}{\bb{C}}$, denoted by $A_1 \oplus \dots \oplus A_k$, or $\bigoplus_{i=1}^k A_i$, is the $n$-by-$n$ matrix 
\[ 
 \left[
 \begin{array}{ccc}
 A_1 &  & \multirow{2}{*}{\Large 0} \\
 \multirow{2}{*}{\Large 0} & \ddots &  \\
  &  & A_k
 \end{array}
 \right], 
\]
where $n = \sum_{i=1}^k n_i$. The \emph{Kronecker product} of $A \in \mat{n,n}{\bb{F}}$ and $B \in \mat{p,q}{\bb{F}}$, denoted by $A \otimes B$, is the $mp$-by-$nq$ matrix defined by 
\[ A \otimes B = 
\begin{bmatrix} 
a_{11} B & \cdots & a_{1n} B 	\\ 
\vdots & \ddots & \vdots 		\\ 
a_{m1} B & \cdots & a_{mn} B 
\end{bmatrix}.\]

For $S \in \gl{n}{\bb{C}}$, the \emph{relative gain array of $S$}, denoted by $\Phi(S)$, is defined by $\Phi(S) = S \circ S^{-\top}$, where $S^{-\top} := (\inv{S})^\top = \inv{(S^\top)}$. It is well-known (see, e.g., \cite{js1986}) that if $A = S D_x \inv{S}$, then 
\begin{align} \Phi(S)(x) = \Diag{A} \label{rgadiag}. \end{align}

For the following, the size of each matrix will be clear from the context in which it appears:
\begin{itemize}
\item $I$ denotes the identity matrix;  
\item $e_i$ denotes the $i\textsuperscript{th}$-column of $I$; 
\item $e$ denotes the all-ones vector; 
\item $J$ denotes the all-ones matrix, i.e., $J = e e^\top$; and 
\item $K$ denotes the \emph{exchange} matrix, i.e., $K = [ e_n|\cdots|e_1 ]$.
\end{itemize}

A matrix is called \emph{symmetric}, \emph{persymmetric}, or \emph{centrosymmetric} if $A - A^\top = 0$, $AK - KA^\top = 0$, or $AK - KA = 0$, respectively. A matrix possessing any two of these symmetry conditions can be shown to possess the third, thus we define a \emph{trisymmetric} matrix as any matrix possessing two of the aforementioned properties.  

If $A$ is an entrywise nonnegative (positive) matrix, then we write $A \geq 0$ ($A > 0$, respectively). An $n$-by-$n$ nonnegative matrix $A$ is called \emph{(row) stochastic} if $\sum_{j=1}^n a_{ij} = 1,~\forall~i \in \bracket{n}$; \emph{column stochastic} if $\sum_{i=1}^n a_{ij} = 1,~\forall~j\in \bracket{n}$; and \emph{doubly stochastic} if it is row stochastic and column stochastic.

We recall the well-known Perron-Frobenius theorem (see, e.g. \cite{bp1994}, \cite[Chapter 8]{hj1990}, or \cite{m1988}).

%-------------------------------------------------------------------------------------------------------------------------------------------------------------------------------------
\begin{thm}[Perron-Frobenius]
\label{pft}
If $A \geq 0$, then $\rho \in \sigma$, and there is a nonnegative vector $x$ such that $Ax = \rho x$.  
\end{thm}

\begin{rem}
The scalar $\rho$ is called the \emph{Perron root} of $A$. When $\sum_{i=1}^n x_i = 1$, $x$ is called the \emph{(right) Perron vector of $A$} and the pair $(\rho, x)$ is called the \emph{Perron eigenpair of $A$}. 

Since the nonnegativity of $A$ is necessary and sufficient for the nonnegativity of $A^\top$, if $A \geq 0$, then, following \hyp{Theorem}{pft}, there is a nonnegative vector $y$, such that $y^\top A = \rho y^\top$. When $\sum_{i=1}^n y_i = 1$, $y$ is called the \emph{left Perron vector of $A$}.
\end{rem}

Given vectors $v_1, \dots, v_n \in \bb{R}^n$ and scalars $\alpha_1,\dots, \alpha_n \in \bb{R}$, the linear combination $\sum_{i=1}^n \alpha_i v_i$ is called a \emph{conical combination} if $\alpha_i \geq 0$ for every $i \in \bracket{n}$; and a \emph{convex combination} if, in addition, $\sum_{i=1}^n \alpha_i = 1$. The \emph{conical hull} of the vectors $v_1, \dots, v_n$, denoted by $\coni{v_1, \dots, v_n}$, is the set of all conical combinations of the vectors, and the \emph{convex hull}, denoted by $\conv{v_1, \dots, v_n}$, is the set of all convex combinations of the vectors. 

Given $A \in \mat{m,n}{\bb{R}}$ and $b \in \bb{R}^m$, the \emph{polyhedron determined by $A$ and $b$} is the set $\mathcal{P}(A,b) := \{ x \in \bb{R}^n : Ax \leq b \}$. When $b=0$, $\mathcal{P}(A,0)$ is called the \emph{polyhedral cone determined by $A$} and is denoted by $\mathcal{C}(A)$. Lastly, recall that a \emph{polytope} is a bounded polyhedron. We say that $S \subseteq \bb{R}^n$ is \emph{polyhedral} if $S=\mathcal{P}(A,b)$ for some $m \times n$ matrix $A$ and $b \in \bb{R}^m$.

%A \emph{directed graph} (or \emph{digraph}) $\Gamma = (V,E)$ consists of a finite set $V$ of elements called \emph{vertices}, together with a set $E \subseteq V \times V$ of ordered pairs called \emph{arcs}. A sequence of $m$ successively adjacent edges 
%\[ (a_0,a_1),\dots,(a_{m-1},a_m),~m>0 \] 
%is called a \emph{walk}, and a closed walk with distinct edges and vertices is called a \emph{cycle}. Two vertices $v$ and $w$ are called \emph{strongly connected} if there exists a walk from $v$ to $w$ and from $w$ to $v$ (following \cite{br1991}, a vertex is regarded as strongly connected to itself). Strong connectivity defines an equivalence relation on the vertices of $\Gamma$ yielding the partition $V = \bigcup_{k=1}^t V_k$. The subdigraphs $\Gamma_k = (V_k, E_k)$, formed by taking vertices in $V_k$ and arcs incident to them, are called the \emph{strong components} of $\Gamma$. The digraph $\Gamma$ is \emph{strongly connected} if it has exactly one strong component, i.e., if every pair of vertices is strongly connected. For $A \in \mat{n}{\bb{C}}$, the \emph{directed graph} (or \emph{digraph}) of $A$, denoted by $\Gamma = \dg{A}$, has vertex set $V = \langle n \rangle$ and arc set $E = \{ (i, j) \in V \times V : a_{ij} \neq 0\}$.

%----------------------------------------------------------------------------------------------------------------
\section{Spectrahedral Sets and Perron Similarities}
%----------------------------------------------------------------------------------------------------------------

%-------------------------------------------------------------------------------------------------------------------------------------------------------------------------------------
\begin{mydef}
For $S \in \gl{n}{\bb{R}}$, let $\cone{S} := \{x \in \bb{R}^n : S D_x \inv{S} \geq 0 \}$ and $\mathcal{A}(S) := \{ A \in \mat{n}{\bb{R}} : A = S D_x \inv{S},~x \in \cone{S} \}$. 		 
\end{mydef}

%-------------------------------------------------------------------------------------------------------------------------------------------------------------------------------------
\begin{rem}
Since $S I \inv{S} = I \geq 0$ for every invertible matrix $S$, it follows that the sets $\cone{S}$  and $\mathcal{A}(S)$ are always nonempty. Specifically, $\coni{e} \subseteq \cone{S}$ for every invertible matrix $S$. In the sequel, we will state a necessary and sufficient condition on $S$ such that $\coni{e} \subset \cone{S}$. Moreover, if $\alpha$, $\beta \geq 0$ and $x$, $y \in \cone{S}$, then $\alpha x + \beta y \in \cone{S}$ so that $\cone{S}$ is a \emph{convex cone}. We refer to $\cone{S}$ as the \emph{Perron spectracone of $S$}. 

The convex cone $\mathcal{A}(S)$ is a nonnegative commutative algebra. In \hyp{Section}{hadspec}, we will show that if $H_n$ is the \emph{Walsh matrix} of order $2^n$, then $\mathcal{A}(H_n)$ is a nonnegative \emph{Bose-Mesner algebra}.  
\end{rem}

Before cataloging basic properties of the spectracone, we require the following lemma.

%-------------------------------------------------------------------------------------------------------------------------------------------------------------------------------------
\begin{lem} 
\label{permdiag}
If $P$ is a permutation matrix and $x \in \bb{C}^n$, then $P D_x P^\top = D_y$, where $y = Px$.
\end{lem}

\begin{proof}
Because a permutation similarity effects a simultaneous permutation of the rows and columns of a matrix, it follows that $P D_x P^\top$ is diagonal, say $D_y$. Following \eqref{rgadiag}, 
\[ y = \Phi(P)(x) = \left[ P \circ (\inv{P})^\top \right]x = Px. \qedhere \] 
\end{proof}

%-------------------------------------------------------------------------------------------------------------------------------------------------------------------------------------
\begin{prop} \label{propcone}
If $S \in \gl{n}{\bb{R}}$ and $P$ is a permutation matrix, then  
\begin{enumerate}[label=(\roman*)]
\item $\cone{S}$ is a polyhedral cone; 
\item $\mathcal{C}(SP)= P^\top \cone{S} := \{ y \in \bb{R}^n: y= P^\top x,~ x\in \cone{S} \}$;   
\item $\mathcal{C}(P S)= \cone{S}$;
\item $\mathcal{C}(D_v S) = \cone{S} $ for any $v > 0$; 
\item $\mathcal{C}(S D_v) = \cone{S}$ for any nonzero $v \in \bb{R}^n$; and
\item $\cone{S} = \mathcal{C}(\inv{(S^\top)})$. 
\end{enumerate}
\end{prop}

\begin{proof} The proofs of parts (iii) -- (vi) are straightforward exercises; thus, we provide proofs for parts (i) and (ii):  
\begin{enumerate}[label=(\roman*)]
% (i)
\item The matricial inequality $S D_x \inv{S} \geq 0$ specifies $n^2$ linear homogeneous inequalities in the variables $x_1,\dots,x_n$; specifically, if $S = [ s_{ij} ]$ and $\inv{S} = [t_{ij}]$, then the $(i,j)$-entry of $S D_x \inv{S}$ is $\sum_{k=1}^n s_{ik} t_{kj} x_k$. If $\vz{\cdot}$ denotes columnwise vectorization, then  
\begin{align}
S D_x \inv{S} \geq 0 \Longleftrightarrow \vz{S D_x \inv{S}} \geq 0 \Longleftrightarrow
-\begin{bmatrix}
s_1 \circ t_1 \\
\vdots 	\\
s_n \circ t_1 \\
s_1 \circ t_2 \\
\vdots \\
s_n \circ t_2 \\
s_1 \circ t_n \\
\vdots \\
s_n \circ t_n
\end{bmatrix} x \leq 0, \label{polymatrix}
\end{align}  
where $s_i$ and $t_i$ denote the $i\textsuperscript{th}$-row and $i\textsuperscript{th}$-column of $S$ and $\inv{S}$, respectively. 

% (iii)
\item Following Lemma \hyperref[permdiag]{\ref*{permdiag}},  
\begin{align*} 
y \in \mathcal{C}(SP) 
&\Longleftrightarrow SP D_y P^\top \inv{S} \geq 0 							\\
&\Longleftrightarrow S D_x \inv{S} \geq 0, ~x = Py 							\\ 
&\Longleftrightarrow x \in \cone{S} 								\\
&\Longleftrightarrow y \in P^\top \cone{S}. \qedhere
\end{align*}

%% (ii)
%\item Note that 
%\[ x \in \cone{S} \Longleftrightarrow S D_x \inv{S} \geq 0 \Longleftrightarrow P S D_x \inv{S} P^\top \geq 0 \Longleftrightarrow x \in \mathcal{C}(PS). \]
%
%% (iv)
%\item If $v > 0$, then
%\[ x \in \cone{S} \Longleftrightarrow S D_x \inv{S} \geq 0 \Longleftrightarrow D_v S D_x \inv{S} \inv{D_v} \geq 0 \Longleftrightarrow x \in \mathcal{C}(D_v S). \]
%
%% (v)
%\item If $v \in \bb{R}^n$, $v \neq 0$, then
%\[ x \in \cone{S} \Longleftrightarrow S D_x \inv{S} \geq 0 \Longleftrightarrow  S D_v D_x \inv{D_v} \inv{S} \geq 0 \Longleftrightarrow x \in \mathcal{C}(S D_v). \]
%
%% (vi)
%\item Note that
%\begin{align*} 
%x \in \cone{S} \Longleftrightarrow S D_x \inv{S} \geq 0 &\Longleftrightarrow  (S D_x \inv{S})^\top \geq 0 \\
%&\Longleftrightarrow  \inv{(S^\top)} D_x S^\top \geq 0 									\\
%&\Longleftrightarrow x \in \mathcal{C}\left( \inv{(S^\top)} \right). 
%\end{align*}
\end{enumerate}
\end{proof}

Let  $\mathcal{B}^n = \{ x \in \bb{R}^n: || x ||_\infty \leq 1 \}$. For $k \in \langle n \rangle$, let $P_k$ be the $(n-1)$-by-$n$ matrix obtained by deleting the $k\textsuperscript{th}$-row of $I$, and define $\pi_k : \bb{R}^n \longrightarrow \bb{R}^{n-1}$ by $\pi_k (x) = P_k x$.

%------------
\begin{mydef} For $S \in \gl{n}{\bb{R}}$, let $\mathcal{W}(S) := \cone{S} \cap \mathcal{B}^n$, $\tope{S} := \{x \in \cone{S} : x_1 = 1 \}$, and $\mathcal{P}^1(S) := \{ y \in \bb{R}^{n-1} : y = \pi_1(x),~x \in \tope{S} \}$. 		
\end{mydef}

%------------
\begin{rem}
Since $S I \inv{S} = I \geq 0$ for every invertible matrix $S$, it follows that the sets $\mathcal{W}(S)$ and $\tope{S}$ are always nonempty; if $n \geq 2$, then $\mathcal{P}^1(S)$ is always nonempty.
\end{rem}

%-------------------------------------------------------------------------------------------------------------------------------------------------------------------------------------
\begin{prop}
If $S \in \gl{n}{\bb{R}}$, then the sets $\mathcal{W}(S)$, $\tope{S}$, and $\mathcal{P}^1(S)$ are polytopes. 
\end{prop}

\begin{proof}
A polyhedral description of $\mathcal{W} (S)$ is obtained by appending the inequalities 
\begin{align}
\left\{ 
\begin{array}{rl}
x_i \leq 1, & i \in \bracket{n} 	\\
-x_i \leq 1, & i \in \bracket{n} 
\end{array} \right. \label{auxineq}
\end{align}
to \eqref{polymatrix}. Since $\mathcal{W} (S) \subset \mathcal{B}^n$, it follows that it is bounded and hence a polytope. 

Similarly, appending the  inequalities $x_1 \leq 1$ and $-x_1 \leq -1$ to \eqref{polymatrix} and \eqref{auxineq} yields a half-space description of $\tope{S}$. Since $\tope{S} \subset \mathcal{B}^n$, it follows that it is bounded and hence a polytope.   

It can be shown via \emph{Fourier-Motzkin elimination} that the projection of a polyhedron is a polyhedron (see, e.g., \cite{d2007} and references therein). Thus, $\mathcal{P}^1(S)$ is a polytope. 
\end{proof}

%-------------------------------------------------------------------------------------------------------------------------------------------------------------------------------------
\begin{rem}
For $S \in \gl{n}{\bb{R}}$, we refer to $\tope{S}$ as the \emph{Perron spectratope} of $S$. 
\end{rem}

%-------------------------------------------------------------------------------------------------------------------------------------------------------------------------------------
\begin{prop}
\label{propcartprod}
If $S = T \oplus U \in \gl{n}{\bb{R}}$, then
\begin{enumerate}[label=(\roman*)]
\item $\cone{S} = \mathcal{C}(T) \times \mathcal{C}(U)$;
\item$\tope{S} = \mathcal{P}(T) \times \mathcal{W}(U)$; and
\item $\mathcal{P}^1(S) = \mathcal{P}^1(T) \times \mathcal{W}(U)$. 
\end{enumerate}
\end{prop}

\begin{proof}
All three parts are straightforward exercises.
\end{proof}

Recall that a \emph{scalar matrix} is any matrix of the form $A = \alpha I$, $\alpha \in \bb{F}$. Let $S \in \gl{n}{\bb{R}}$ and suppose there is a real diagonal matrix $D$ and a nonnegative, nonscalar matrix $A$ such that $A = S D \inv{S}$. Following \hyp{Theorem}{pft}, there is an $i \in \langle n \rangle$ such that $S e_i$ and $e_i^\top \inv{S}$ are both nonnegative (or both nonpositive). In consideration of part (v) of Proposition \hyperref[propcone]{\ref*{propcone}}, we may assume that they are both nonnegative. This motivates the following definition.

%-------------------------------------------------------------------------------------------------------------------------------------------------------------------------------------
\begin{mydef}
We call an invertible matrix $S$ a \emph{Perron-similarity} if there is an $i \in \bracket{n}$ such that $S e_i$ and $e_i^\top \inv{S}$ are nonnegative.
\end{mydef}

Given $S \in \gl{n}{\bb{R}}$, it is natural to determine necessary and sufficient conditions so that $S$ is a Perron-similarity; to that end, we require the following theorem. 

%-------------------------------------------------------------------------------------------------------------------------------------------------------------------------------------
\begin{thm}
\label{thm:nnrv}
Let 
\[ S = \left[ \begin{array}{c}  s_1^\top \\ \vdots \\  s_n^\top \end{array} \right] \in GL_{n} (\bb{R}), \]  
where $s_k \in \bb{R}^n$, for every $k \in \bracket{n}$. If $y^\top := e_i^\top \inv{S}$, then $y \geq 0$ if and only if $e_i \in \coni{s_1,\dots,s_n}$. 
Moreover, $y > 0$ if and only if $e_i \in \interior{\coni{s_1,\dots,s_n}}$. 
\end{thm}

\begin{proof}
If $y \geq 0$, then 
\begin{align}  
e_i = S^\top y = \sum_{k=1}^n y_k s_k  \in \coni{s_1, \dots, s_n} \label{e1conihull}
\end{align}
If $y > 0$, then \eqref{e1conihull} implies $e_1 \in \interior{\coni{s_1, \dots, s_n}}$. 

Conversely, if $e_i \in \coni{s_1, \dots, s_n}$, then there exist nonnegative scalars $\lambda_1,\dots,\lambda_n$ such that 
\[ e_i = \sum_{k=1}^n \lambda_k s_k= S^\top \lambda, \]
where $\lambda = \left[ \lambda_1 ~ \cdots ~ \lambda_n \right]^\top$. By hypothesis, $S^\top y = e_i$, and since $S^\top$ is invertible, it follows that $y = \lambda \geq 0$. Lastly, if $e_i \in \interior{\coni{s_1,\dots, s_n}}$, then $y = \lambda > 0$.
\end{proof}

%-------------------------------------------------------------------------------------------------------------------------------------------------------------------------------------
\begin{cor}
Let $S \in \gl{n}{\bb{R}}$ and suppose that
\[ S = \left[ \begin{array}{c}  s_1^\top \\ \vdots \\  s_n^\top \end{array} \right] \mbox{ and } (\inv{S})^\top = \left[ \begin{array}{c}  t_1^\top \\ \vdots \\  t_n^\top \end{array} \right].\]
Then $S$ is a Perron-similarity if and only if there is an $i \in \bracket{n}$ such that $e_i \in \coni{s_1,\dots, s_n}$ and $e_i \in \coni{t_1,\dots, t_n}$.
\end{cor}

%-------------------------------------------------------------------------------------------------------------------------------------------------------------------------------------
\begin{cor}
If $Q \in \og{n}$, then $Q$ is a Perron-similarity if and only if there is an $i \in \bracket{n}$ such that $Qe_i \geq 0$.
\end{cor}

%-------------------------------------------------------------------------------------------------------------------------------------------------------------------------------------
\begin{rem}
It is well-known that a nonnegative matrix $A$ is stochastic if and only if $Ae = e$ (\cite[\S 8.7, Problem 4]{hj1990}). Thus, $A$ is doubly stochastic if and only if $Ae = e$ and $e^\top A = e^\top$. If there are real scalars $\alpha$ and $\beta$ such that $Se_i = \alpha e$ and $e_i^\top \inv{S} = \beta e^\top$, then $\tope{S}$ contains spectra that are doubly stochastically realizable; however, notice that the converse does not hold: for example, if 
\begin{equation*}
S= 
\begin{bmatrix}
1 & 1 & 0	\\
1 & -1 & 0	\\
0 & 0 & 1	
\end{bmatrix},
\end{equation*}
and $v = \begin{bmatrix} 1 & \lambda & 1 \end{bmatrix}^\top$, where $\lambda \in [-1,1]$, then $M = S D_v \inv{S}$ is doubly stochastic. 

If $Q \in \og{n}$, then $\tope{Q}$ contains spectra that are symmetrically, doubly stochastically realizable if and only if $Qe_i = e$ and $e_i^\top Q = e^\top$.   
\end{rem}

%-------------------------------------------------------------------------------------------------------------------------------------------------------------------------------------
\begin{ex}
Although the nonsingular matrix
\[
S =
\begin{bmatrix}
1 & 1/2 	\\
1 & 1
\end{bmatrix}
\]
has two positive columns, it is not a Perron-similarity; indeed, note that 
\[
\inv{S} =
\left[ \begin{array}{rr}
2 & -1 	\\
-2 & 2
\end{array} \right].
\]
\end{ex}

%-------------------------------------------------------------------------------------------------------------------------------------------------------------------------------------
\begin{rem}
We briefly digress to make the following observation: recall that a nonsingular $M$-matrix is any matrix of the form $S = \alpha I - T$, where $T \geq 0$ and $\alpha > \sr{T}$. It is well-known that $S$ is an $M$-matrix if and only if $S^{-1} \geq 0$ (Plemmons lists forty characterizations in \cite{p1977}). Thus, following Theorem \hyperref[thm:nnrv]{\ref*{thm:nnrv}}, $S$ is an $M$-matrix if and only if $e_i \in \coni{s_1,\dots,s_n}$ for every $i \in \bracket{n}$. 
\end{rem}

%-------------------------------------------------------------------------------------------------------------------------------------------------------------------------------------
\begin{cor}
If $S \in \gl{n}{\bb{R}}$, then $\cone{S} \backslash \coni{e} \neq \emptyset$ if and only if $S$ is a Perron-similarity. 
\end{cor}

\begin{proof}
If $S$ is a Perron-similarity, then there is an $i \in \bracket{n}$ such that the vectors $x := Se_i$ and $y^\top := e_i^\top S$ are nonnegative. Thus, $S D_{e_i} \inv{S} = xy^\top \geq 0$, so that $e_i \in \cone{S}$.

Conversely, if $\cone{S} \backslash \coni{e} \neq \emptyset$, then then there is a vector $x \neq e$ such that $S D_x \inv{S} \geq 0$. The result now follows from \hyperref[pft]{\ref*{pft}}.     
\end{proof}

%-------------------------------------------------------------------------------------------------------------------------------------------------------------------------------------
\section{RNIEP and SNIEP for Low Dimensions}
%-------------------------------------------------------------------------------------------------------------------------------------------------------------------------------------

In this section, we verify the known necessary and sufficient conditions for the SNIEP and RNIEP when $n \leq 4$. 

For $n \geq 2$, notice that, following \eqref{trnn} and \eqref{sprad}, 
\[ \bigcup_{S \in \gl{n}{\bb{R}}} \mathcal{P}^1(S) \subseteq \mathcal{T}^{n-1}, \]
where $\mathcal{T}^{n-1} := \{ x \in \mathcal{B}^{n-1}: 1 + e^\top x \geq 0 \}$. The region $\mathcal{T}^{n-1}$ is known as the \emph{trace-nonnegative polytope} \cite{kn2001}.

%-----------------------------------------------------------------------
\begin{thm}
If $\sigma = \{ \lambda_1, \dots, \lambda_n \}$ and $n \leq 4$, then $\sigma$ is realizable if and only if $\sigma$ satisfies \eqref{trnn} and \eqref{sprad}. Futhermore, the realizing matrix can be taken to be symmetric.
\end{thm}

\begin{proof} We give a proof for $2 \leq n \leq 4$, given that the result is trivial when $n=1$ (however, note that the similarity $H_0 := 1$ yields all possible spectra).   

\emph{Case $n =2$}.  Figure \hyperref[rniepnequalstwo]{\ref*{rniepnequalstwo}} depicts the spectrahedral sets for the matrix
\[
H_1 =
\begin{bmatrix}
1 & 1 	\\
1 & -1
\end{bmatrix},
\]
which are established in Theorem \hyperref[walshcone]{\ref*{walshcone}} and Corollary \hyperref[walshspec]{\ref*{walshspec}}. This solves the SNIEP and RNIEP when $n=2$, since $\mathcal{P}^1(H_1) = \mathcal{B}^1 = \mathcal{T}^1$ and the realizing matrix is 
\[ 
\frac{1}{2}
\begin{bmatrix} 
\lambda_1+\lambda_2 & \lambda_1 - \lambda_2 \\ 
\lambda_1 - \lambda_2 & \lambda_1 + \lambda_2 
\end{bmatrix},~\lambda_1 \geq \lambda_2. \]

\begin{figure}[H]
\centering
\begin{tikzpicture}
\begin{axis}[
xticklabel style={font=\scriptsize},
yticklabel style={font=\scriptsize},
xmin=-2,
xmax=2,
ymin=-2,
ymax=2,
y label style={rotate=-90},
xlabel=$x_1$,
ylabel=$x_2$]

% P_0(H_1)
\addplot[domain=0:1, name path=A] {x} {};
\addplot[domain=0:1, name path=B] {-x}{};
\addplot[Gray!50] fill between[of=A and B];
\node[anchor=center] at (.6,0) {$\mathcal{W}(H_1)$};

% C(H_1)
\addplot[domain=1:2, name path=C] {x} {};
\addplot[domain=1:2, name path=D] {-x}{};
\addplot[Gray!25] fill between[of=C and D];
\node[anchor=center] at (1.6,0) {$\mathcal{C}(H_1)$};

% P(H_1)
\addplot[very thick] coordinates {(1,1) (1,-1)};
\node[pin=225:$\mathcal{P}(H_1)$] at (1.05,-.95) {};

% P^1(H_1)
\addplot[very thick] coordinates {(0,-1) (0,1)};
\node[pin=135:$\mathcal{P}^1(H_1)$] at (.05,.95) {};
\end{axis}
\end{tikzpicture}
\caption{RNIEP \& SNIEP for $n=2$.}
\label{rniepnequalstwo}
\end{figure}
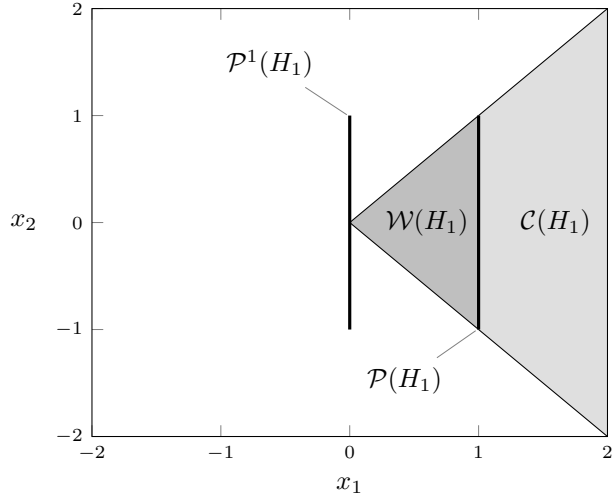

\emph{Case $n = 3$}. Let 
\[ 
S := H_1 \oplus H_0 = 
\begin{bmatrix}
1 & 1 & 0	\\
1 & -1 & 0	\\
0 & 0 & 1	
\end{bmatrix}
~\mbox{and}~ 
P :=
\begin{bmatrix}
1 & 0 & 0	\\
0 & 0 & 1	\\
0 & 1 & 0
\end{bmatrix}. \]
For $a \in [0,1]$, let $b := 1 - a$ and
\[
S_a := 
\begin{bmatrix}
1 & 1 & 0	\\
1 & -a & 1	\\
1 & -a & -1	
\end{bmatrix}.\]
\hyp{Figure}{rniepnequalsthree} depicts the projected Perron spectratopes for the matrices $S$, $SP$, and $S_a$.

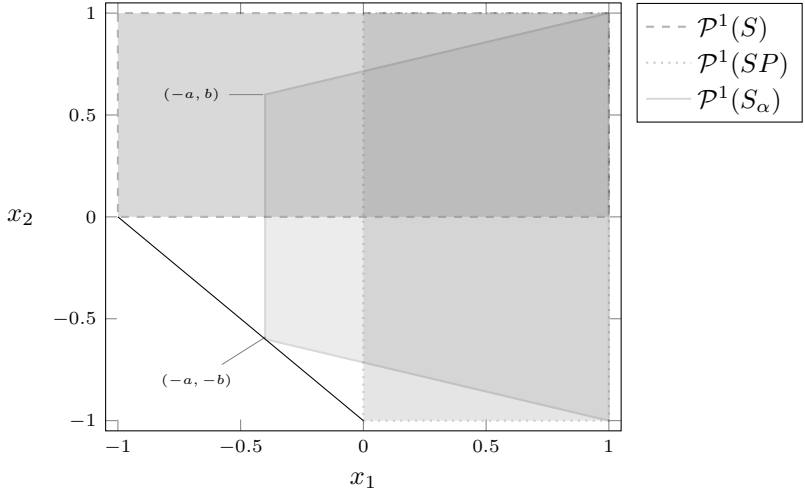
\begin{figure}[H]
\centering
\begin{tikzpicture}
\begin{axis}[
xticklabel style={font=\scriptsize},
yticklabel style={font=\scriptsize},
xmin=-1.05,
xmax=1.05,
ymin=-1.05,
ymax=1.05,
y label style={rotate=-90},
xlabel=$x_1$,
ylabel=$x_2$,
legend entries={$\mathcal{P}^1(S)$,$\mathcal{P}^1(SP)$,$\mathcal{P}^1(S_\alpha)$},
legend style={nodes=right},
legend pos= outer north east]

\addplot[thick,dashed,fill=Gray,opacity=0.3] coordinates{(1,1) (-1,1) (-1,0) (1,0)(1,1)};

\addplot[thick,dotted, fill=Gray,opacity=0.2] coordinates{(1,1) (1,-1)(0,-1)(0,1)(1,1)};

\addplot[thick, fill=Gray,opacity=0.15] coordinates{(1,1) (-.4,.6)(-.4,-.6)(1,-1)(1,1)};
\node[pin=225:{\tiny $(-a,-b)$}] at (-.36,-.56) {};
\node[pin=180:{\tiny $(-a,b)$}] at (-.37,.6) {};

\addplot[black] coordinates{(-1,0)(0,-1)};
\end{axis}
\end{tikzpicture}
\caption{RNIEP \& SNIEP for $n=3$.}
\label{rniepnequalsthree}
\end{figure}

A straightforward calculation reveals that if $v = [1~x~y]^\top$, then the matrix $S_a D_v \inv{S_a}$ is given by
\[ 
\frac{1}{2(1+a)}
\begin{bmatrix}
2(x + a) &             1 - x &             1 - x    			\\
2a(1 - x) & a(x + y) + y + 1 & a(x - y) - y + 1    \\
2a(1 - x) & a(x - y) - y + 1 & a(x + y) + y + 1
\end{bmatrix}. \]
Furthermore, if $u = [1~\sqrt{2a}~\sqrt{2a}]^\top$, then the matrix $\inv{D_u} S_a D_v \inv{S_a} D_u$  is given by
\[ 
\frac{1}{2(1+a)}
\begin{bmatrix}
2(x + a) &             \sqrt{2a}(1 - x) &             \sqrt{2a}(1 - x)    			\\
\sqrt{2a}(1 - x) & a(x + y) + y + 1 & a(x - y) - y + 1    \\
\sqrt{2a}(1 - x) & a(x - y) - y + 1 & a(x + y) + y + 1
\end{bmatrix}, \]
thus, the realizing matrix can by taken to be symmetric.

\emph{Case $n=4$}. Without loss of generality, assume that $\lambda_1 \geq \lambda_2 \geq \lambda_3 \geq \lambda_4$. If $v := [\lambda_1~\lambda_4~\lambda_2~\lambda_3]^\top$ and $S := H_1 \oplus H_1$, then
\[ S D_v \inv{S}
=
\frac{1}{2} 
\begin{bmatrix}
\lambda_1+\lambda_4 & \lambda_1 - \lambda_4 & 0 & 0 	\\ 
\lambda_1 - \lambda_4 & \lambda_1 + \lambda_4 & 0 & 0 	\\
0 & 0 & \lambda_2+\lambda_3 & \lambda_2 - \lambda_3 	\\ 
0 & 0 & \lambda_2 - \lambda_3 & \lambda_2 + \lambda_3
\end{bmatrix}. \]
Thus, $\cone{S}$ contains all spectra such that 
\begin{enumerate}[label=(\roman*)]
\item $\lambda_1 \geq \lambda_2 \geq \lambda_3 \geq \lambda_4 \geq 0$; 
\item $\lambda_1 \geq \lambda_2 \geq \lambda_3 \geq 0 > \lambda_4$; or 
\item $\lambda_1 \geq \lambda_2 \geq 0 > \lambda_3 \geq \lambda_4$, and $\lambda_2 + \lambda_3 \geq 0$.
\end{enumerate}

If $H_2 := H_1 \otimes H_1$, then $H_2 D_v \inv{H_2}$ equals 
\[ 
\frac{1}{4}
\begin{bmatrix}
\lambda_1 + \lambda_4 + \lambda_2 + \lambda_3 & 
\lambda_1 - \lambda_4 + \lambda_2 - \lambda_3 & 
\lambda_1 + \lambda_4 - \lambda_2 - \lambda_3 & 
\lambda_1 - \lambda_4 - \lambda_2 + \lambda_3 	\\ 
%2
\lambda_1 - \lambda_4 + \lambda_2 - \lambda_3 & 
\lambda_1 + \lambda_4 + \lambda_2 + \lambda_3 & 
\lambda_1- \lambda_4 - \lambda_2 + \lambda_3 & 
\lambda_1 + \lambda_4 - \lambda_2 - \lambda_3 	\\ 
%3
\lambda_1 + \lambda_4 - \lambda_2 - \lambda_3 & 
\lambda_1 - \lambda_4 - \lambda_2 + \lambda_3 & 
\lambda_1 + \lambda_4 + \lambda_2 + \lambda_3 & 
\lambda_1 - \lambda_4 + \lambda_2 - \lambda_3 	\\ 
%4
\lambda_1 - \lambda_4 - \lambda_2 + \lambda_3 & 
\lambda_1 + \lambda_4 - \lambda_2 - \lambda_3 & 
\lambda_1 - \lambda_4 + \lambda_2 - \lambda_3 & 
\lambda_1 + \lambda_4 + \lambda_2 + \lambda_3 	
\end{bmatrix}. \]
Thus, $\cone{S}$ contains all spectra such that 
\begin{enumerate}[label=(\roman*)]
\item $\lambda_1 \geq \lambda_2 \geq 0 > \lambda_3 \geq \lambda_4$, and $\lambda_2 + \lambda_3 < 0$; or 
\item $\lambda_1 \geq 0 > \lambda_2 \geq \lambda_3 \geq \lambda_4$.
\end{enumerate}
\hyp{Figure}{hadspectwo} depicts the projected Perron spectratope of $H_2$, as established in \hyp{Corollary}{walshspec}.
\end{proof}

\begin{figure}[H]
\centering
\begin{tikzpicture}
\begin{axis}
[view={-35}{20},
xmin=-1.05,
xmax=1.05,
ymin=-1.05,
ymax=1.05,
zmin=-1.05,
zmax=1.05,
z label style={rotate=-90},
xlabel=$x_1$,
ylabel=$x_2$,
zlabel=$x_3$]

% Hadamard
\addplot3[thick,fill=Gray,opacity=.4] coordinates{(-1,-1,1) (-1,1,-1) (1,-1,-1) (-1,-1,1)};
\addplot3[thick,fill=Gray,opacity=.2] coordinates{(1,1,1) (-1,-1,1) (1,-1,-1) (1,1,1)};
\addplot3[thick,fill=Gray,opacity=.1] coordinates{(1,1,1) (-1,-1,1) (-1,1,-1) (1,1,1)};

% Sulei
\addplot3[Gray] coordinates{(-1,0,0) (0,-1,0) (0,0,-1) (-1,0,0)}; 
\addplot3[Gray,dashed] coordinates{(0,0,0) (-1,0,0)};
\addplot3[Gray,dashed] coordinates{(0,0,0) (0,-1,0)};
\addplot3[Gray,dashed] coordinates{(0,0,0) (0,0,-1)};
\node[pin=245:{\tiny Sule\u{\i}manova}] at (.1,.1,-.9) {};
\end{axis}
\end{tikzpicture}
\caption{$\mathcal{P}^1(H_2)$.}
\label{hadspectwo}
\end{figure}
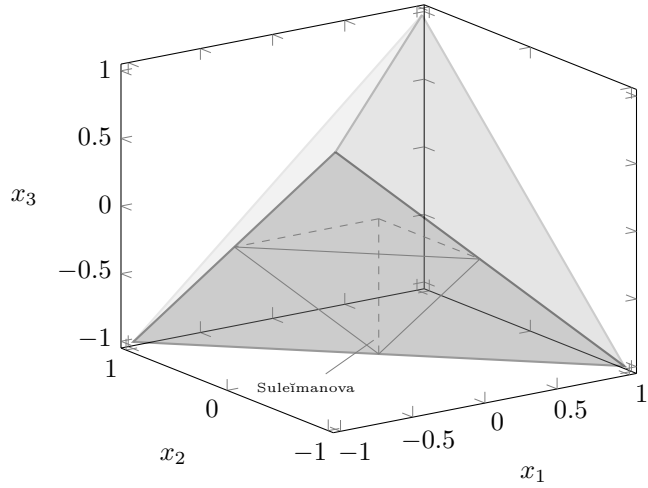

\hyp{Figure}{rniepnequalsthree} suggests that the trace-nonnegative polytope $\mathcal{T}^{2}$ can not be covered by countably many projected Perron spectratopes; this can be proven via the relative gain array.

%-------------------------------------------------------------------------------------------------------------------------------------------------------------------------------------
\begin{lem}
Suppose that $S$ is a $3$-by-$3$ Perron-similarity. If $x$, $y \in \cone{S}$ such that $e^\top x = e^\top y = 0$, then $x = y$.  
\end{lem}

\begin{proof}
If $x \neq y$, then $\rk{\Phi(S)} = 1$; hence 
\[ \Phi(S) = 
\begin{bmatrix}
a & b & c 		\\
ka & kb & kc 		\\
\ell a & \ell b & \ell c
\end{bmatrix}. \]
Since $\Phi(S)$ has row and column sums equal to one (\cite[\S2, Observation 1]{js1986}), it follows that $k = \ell = 1$, and $a = b = c = 1/3$; however, as noted in \cite[\S5, Example 2]{js1986}, the equation $\Phi(X) = (1/3)J$ has no real solutions.
\end{proof}

%-------------------------------------------------------------------------------------------------------------------------------------------------------------------------------------
\begin{cor}
\label{cor:uncount}
If $\{ S_\alpha \}_{\alpha \in \mathcal{I}}$ is a collection of invertible $3$-by-$3$ matrices such that 
\[ \bigcup_{\alpha \in \mathcal{I}} \mathcal{P}^1(S_\alpha) = \mathcal{T}^{2}, \]
then $\mathcal{I}$ is uncountable.  
\end{cor}

%%-------------------------------------------------------------------------------------------------------------------------------------------------------------------------------------
%\begin{rem}
%If \hyp{Corollary}{cor:uncount} holds for $n \geq 5$, then 
%\begin{equation*}
%\bigcup_{S \in \gl{n}{\bb{R}}} \mathcal{P}^1(S)
%\end{equation*}
% may not be polyhedral.
%\end{rem}

%-------------------------------------------------------------------------------------------------------------------------------------------------------------------------------------
\section{Perron Spectratopes of Hadamard Matrices}
\label{hadspec}
%-------------------------------------------------------------------------------------------------------------------------------------------------------------------------------------

Recall that $H$ is a \emph{Hadamard matrix} if $h_{ij} \in \{ \pm 1\}$ and $H H^\top = nI$. If $n$ is a positive integer $n$ such that there is a Hadamard matrix of order $n$, then $n$ is called a \emph{Hadamard order}. The longstanding \emph{Hadamard conjecture} asserts that there is a Hadamard matrix of order $4k$ exists for every $k \in \bb{N}$.

Let $H_0 = [1]$, and for $n \in \bb{N}$, let
\begin{align}
H_n := 
 H_1 \otimes H_{n-1}
=
\left[
\begin{array}{rr}
H_{n-1} & H_{n-1} \\
H_{n-1} & -H_{n-1}
\end{array} \right]
\in \mat{2^n}{\bb{R}}.				\label{hadmtrx}
\end{align}
It is well-known that $H_n$ is a Hadamard matrix for every $n \in \bb{N}_0$, and the construction given in \eqref{hadmtrx} is known as the \emph{Sylvester construction}, and the matrix $H_n$, $n \in \bb{N}_0$ is called the \emph{canonical Hadamard}, or \emph{Walsh} \emph{matrix} of order $2^n$ (for brevity, and given that \emph{Sylvester matrix} is reserved for another matrix, we will use the latter term). 

Walsh matrices satisfy the following additional well-known properties:
\begin{enumerate}[label=(\roman*)]
\item $H_n = H_n^\top$; 
\item $\inv{H_n} = 2^{-n} H_n$;
\item $\trace{H_n} = 0$, $n \geq 1$;  
\item $e_1^\top H_n = e^\top$; 
\item $H_n e_1 =  e$; and
\item $e^\top H_n = ne_1$.
\end{enumerate}
Lastly, note that $H_n$ is a Perron similarity for every $n$. 

%-------------------------------------------------------------------------------------------------------------------------------------------------------------------------------------
\begin{prop} \label{propas}
Let 
\[ 
P_{11} := 
\begin{bmatrix}
1 & 0 \\
0 & 1
\end{bmatrix} \mbox{ and }
P_{12} := 
\begin{bmatrix}
0 & 1 \\
1 & 0
\end{bmatrix}. \] 
For $n \geq 2$, let 
\[ P_{nk} :=
\left\{ 
\begin{array}{rl}
P_{11} \otimes P_{(n-1)k} = 
\begin{bmatrix} 
P_{(n-1)k} & 0 \\
0 & P_{(n-1)k} 
\end{bmatrix} \in \mat{2^n}{\bb{R}}, & k \in \bracket{2^{n-1}}						\\ \\
P_{12} \otimes P_{(n-1)\ell} = 
\begin{bmatrix} 
0 & P_{(n-1)\ell} \\ 
P_{(n-1)\ell} & 0 
\end{bmatrix} \in \mat{2^n}{\bb{R}}, & k \in \bracket{2^n}\backslash \bracket{2^{n-1}},
\end{array} \right. \]
where $\ell = k - 2^{n-1}$.
Then, for $n \in \bb{N}$:
\begin{enumerate}[label=(\roman*)]
\item $P_{n1} = I$; 
\item $P_{n 2^n} = K$;
\item $P_{nk}$ is a trisymmetric permutation matrix, $\forall k \in \langle 2^n \rangle$;
\item $\sum_{k=1}^{2^n} P_{nk} = J$;
\item If $\mathscr{P}_n := \left\{ P_{n1}, \dots, P_{n2^k} \right\}$, then $\{ \mathscr{P}_n, \times \} \cong (\bb{Z}_2)^n$; 
\item $P_{nk}^2 = I$;
\item if $v^\top : = e_k^\top H_n$, then $P_{nk}= 2^{-n} H_n D_{v} H_n$, $\forall k \in \langle 2^n \rangle$;
\item for every $k \neq \ell$, $\left\langle P_{nk}, P_{n\ell} \right\rangle := \trace{P_{nk}^\top P_{n\ell}}=0$; and
\item for every $k$, $\ell \in \bracket{n}$, there is a $j \in \bracket{n}$ such that $P_{nk}  P_{n\ell} =  P_{n\ell} P_{nk} = P_{nj}$.
\end{enumerate}
\end{prop}

\begin{proof}
Parts (i)--(v) follow readily by induction on $n$; part (vi) follows from part (iii); part (viii) follows from parts (iii) and (iv); and part (ix) follows from parts (v) and (vii) (alternately, part (ix) follows from part (vii) and the fact that the rows of $H_n$ form a group with respect to Hadamard product).

For part (vii), we proceed by induction on $n$: for $n=1$, the result follows by a direct computation. 

Assume that the result holds when $n = m > 1$. If $v \in \bb{R}^{2^{m+1}}$, and $u$ and $w$ are the vectors in $\bb{R}^{2^m}$ defined by 
\begin{align} v = \begin{bmatrix} u \\ w \end{bmatrix}, \label{vecsplit} \end{align} 
then $D_v = D_u \oplus D_w$ and  
\begin{align}
H_{m+1} D_v H_{m+1} =
\begin{bmatrix}
H_m D_{u + w} H_m & H_m D_{u - w} H_m \\
H_m D_{u - w} H_m & H_m D_{u + w} H_m
\end{bmatrix}. \label{haddiag}
\end{align}

We distinguish the following cases:
\begin{enumerate}[label=(\roman*)]
\item $k \in \bracket{2^m}$. If $v^\top := e_k^\top H_{m+1}$, then, $u = w = H_m e_k$. Following \eqref{haddiag} and the induction-hypothesis,    
notice that
\begin{align*}
H_{m+1} D_v H_{m+1} 
&= 
\begin{bmatrix}
2 H_m D_u H_m & 0 	\\
0 & 2 H_m D_u H_m
\end{bmatrix}							\\
&= 
\begin{bmatrix}
2 \cdot 2^m P_{mk}  & 0 \\
0 & 2 \cdot 2^m P_{mk} 
\end{bmatrix} = 2^{m+1} P_{(m+1)k},
\end{align*}
and the result is established.

\item $k \in \bracket{2^{m+1}}\backslash \bracket{2^m}$. If $v^\top := e_k^\top H_{m+1}$, then $u = -w = H_m e_\ell$, where $\ell := k - 2^{m-1}$. Following \eqref{haddiag} and the induction-hypothesis, notice that 
\begin{align*}
H_{m+1} D_v H_{m+1} 
&= 
\begin{bmatrix}
0 & 2 H_m D_u H_m  \\
2 H_m D_u H_m & 0
\end{bmatrix}							\\
&= 
\begin{bmatrix}
0 & 2 \cdot 2^m P_{m\ell}   \\
2 \cdot 2^m P_{m\ell}  & 0
\end{bmatrix} 
= 2^{m+1} P_{(m+1)k},
\end{align*}
and the result is established. \qedhere
\end{enumerate}
\end{proof}

%-------------------------------------------------------------------------------------------------------------------------------------------------------------------------------------
\begin{thm}
\label{walshcone}
The Perron spectracone of the Walsh matrix of order $2^n$ is the conical hull of its rows. 
\end{thm}

\begin{proof}
For $x \in \bb{C}^{2^n}$, let $v^\top := x^\top H_n$, and $M = M_x := 2^{-n} H_n D_v H_n$. Since  
\[ v^\top = x^\top H_n = \left( \sum_{k=1}^{2^n} x_k  e_k^\top \right) H_n  = \sum_{k=1}^{2^n} x_k \left( e_k^\top H_n \right) = \sum_{k=1}^{2^n} x_k v_k^\top, \]
where $v_k :=  e_k^\top H_n$, it follows from part (vii) of Proposition \hyperref[propas]{\ref*{propas}} that 
\begin{align*}
M = 
2^{-n} H_n D_v H_n = 
\sum_{k=1}^{2^n} x_k \left( 2^{-n} H_n D_{v_k} H_n \right) = 
\sum_{k=1}^{2^n} x_k P_{nk},
\end{align*}
i.e., $M \in \spn{\mathscr{P}_n}$. 

In view of parts (iii) and (iv) of \hyp{Propostion}{propas}, it follows that the $k\textsuperscript{th}$-entry of $x$ appears exactly once in each row and column of $M$. Thus, $M \geq 0$ if and only if $x \geq 0$, i.e., if and only if $v \in \coni{v_1, \dots, v_{2^n}}$.
\end{proof}

The following result is useful in quickly determining whether a vector belongs to $\cone{H_n}$.

%-------------------------------------------------------------------------------------------------------------------------------------------------------------------------------------
\begin{cor}
If $v \in \bb{R}^{2^n}$, where $v_1 \geq \cdots \geq v_n$, then $v \in \cone{H_n}$ if and only if $H_n v \geq 0$.
\end{cor}

\begin{proof}
If $v \in \cone{H_n}$, then following \hyp{Theorem}{walshcone}, there is a nonnegative vector $x$ such that $v = H_n x$; multipying both sides of this equation by $H_n$ yields $H_n v = 2^n x \geq 0$. 

Conversely, if $x := H_n v \geq 0$, then $H_n x = 2^n v$, i.e., $v = H_n (x/2^n)$; the result now follows from \hyp{Theorem}{walshcone}.
\end{proof}

Parts (i), (iii), (iv), and (ix) of \hyp{Propostion}{propas} demonstrate that $\mathscr{P}_n$ is a $(2^n-1)$-class symmetric (and hence commutative) \emph{association scheme} (see, e.g., the survey \cite{mt2009} and references therein). As a consequence, $\mathcal{A}(H_n)$ is a nonnegative \emph{Bose-Mesner algebra}, i.e., it is closed with respect to matrix transposition, matrix multiplication, and Hadamard product. 

A straightforward proof by induction shows that 
\[
M_x = 
\begin{bmatrix}
x^\top P_{n1} 	\\
\vdots 		\\
x^\top P_{n2^k}
\end{bmatrix} 
=
\begin{bmatrix}
P_{n1} x & \cdots & P_{n2^k}x 
\end{bmatrix}. 
 \]
For example, if $x \in \bb{C}^8$, then 
\[ M_x =
\left[ \begin{array}{*{4}{c}|*{4}{c}} 
x_1 & x_2 & x_3 & x_4 & x_5 & x_6 & x_7 & x_8	\\ 
x_2 & x_1 & x_4 & x_3 & x_6 & x_5 & x_8 & x_7	\\ 
x_3 & x_4 & x_1 & x_2 & x_7 & x_8 & x_5 & x_6	\\
x_4 & x_3 & x_2 & x_1 & x_8 & x_7 & x_6 & x_5	\\
\hline
x_5 & x_6 & x_7 & x_8 & x_1 & x_2 & x_3 & x_4	\\ 
x_6 & x_5 & x_8 & x_7 & x_2 & x_1 & x_4 & x_3	\\ 
x_7 & x_8 & x_5 & x_6 & x_3 & x_4 & x_1 & x_2	\\
x_8 & x_7 & x_6 & x_5 & x_4 & x_3 & x_2 & x_1
\end{array}\right]. \]
Moreover, if $x$, $y \in \bb{C}$, then $M_x \circ M_y = M_{x \circ y}$ and $M_x M_y = M_z$, where 
\[ z = 
\begin{bmatrix} 
x^\top P_{n1}y	\\ 
x^\top P_{n2}y 	\\ 
\vdots 		\\ 
x^\top P_{n2^k}y 
\end{bmatrix} =
\begin{bmatrix} 
x^\top y 		\\ 
x^\top P_{n2}y 	\\ 
\vdots 		\\ 
x^\top Ky 
\end{bmatrix}. \]

%-------------------------------------------------------------------------------------------------------------------------------------------------------------------------------------
\begin{cor}
\label{walshspec}
The Perron spectratope of the Walsh matrix of order $2^n$ is the convex hull of its rows. 
\end{cor}

The import of \hyp{Corollary}{walshspec} can be viewed through the following lens: given an affinely independent set $V = \{ v_1,\dots,v_{n+1} \} \in \bb{R}^n$, recall that the \emph{n-simplex} of $V$ is the set $\mathcal{S}(V) := \conv{V}$. It is well-known (see, e.g., \cite{s1966}) that 
\begin{equation}
\vol{\mathcal{S}(V)} = \left| \frac{1}{n!} \det{M} \right|, \label{simplexvol}
\end{equation} 
where 
\[ M = 
\begin{bmatrix}
1 & v_1^\top 	\\
\vdots & \vdots 	\\
1 & v_{n+1}^\top
\end{bmatrix}. \]
Thus, 
\[ \vol{\mathcal{W} (H_n)} = \frac{1}{(2^n)!}2^{n2^{n-1}}~\mbox{and}~
\vol{\mathcal{P}^1 (H_n)} = \frac{1}{(2^n-1)!}2^{n2^{n-1}}. \]

We call a nonsingular matrix $S$ a \emph{strong Perron similarity} if there is a unique $i \in \bracket{n}$ such that $Se_i > 0$ and $e_i^\top \inv{S} >0$. If $S \in \gl{n}{\bb{R}}$ is a strong Perron similarity, then, without loss of generality, $S = [e~s_2~\cdots~s_n]$, where $|| s_i ||_\infty = 1$. 

Since a Hadamard matrix has maximal determinant among matrices whose entries are less than or equal to 1 in absolute value, it follows that if $\mathcal{P}^1(S) = \mathcal{S}(s_2,\dots,s_{2^n})$, then $\vol{\mathcal{P}^1(S)} \leq \vol{\mathcal{P}^1 (H_n)}$. It is an open question whether $\mathcal{P}^1 (H_n)$ has maximal volume for all Perron spectratopes.

Moreover, \hyp{Corollary}{walshspec} does not seem to hold for general Hadamard matrices. If $H$ is a Hadamard matrix, then any matrix resulting from negating its rows or columns is also a Hadamard matrix. Indeed, if $u$ denotes the $i\textsuperscript{th}$-column of $H$ and $w$ denotes its $i\textsuperscript{th}$-row, then the $i\textsuperscript{th}$-row and $i\textsuperscript{th}$-column of the Hadamard matrix $\hat{H} = D_{h_{ii}u} H D_{w}$ are positive. Thus, without loss of generaltiy, we assume that the first row and column of any Hadamard matrix are positive and refer to such a matrix as a \emph{normalized Hadamard matrix}. 

It can be verified via the MATLAB-command \texttt{Hadamard(12)} that only the first row of the normalized Hadamard matrix of order twelve belongs to its Perron spectratope. However, it is clear every row of a normalized Hadamard matrix is realizable since every row (sans the first) contains an equal number of positive and negative entries and thus is realizable (after a permutation) by the Perron similarity $\bigoplus_{i=1}^{n/2} H_1$.

%-----------------------------------------------------------------------------
\section{Sule\u{\i}manova spectra, the DS-RNIEP, and the DS-SNIEP}
%-----------------------------------------------------------------------------

We begin with the following definition.

\begin{mydef}
We call $\sigma = \{ \lambda_1, \dots, \lambda_n \} \subset \bb{R}$ a \emph{Sule\u{\i}manova spectrum} if $s_1 (\sigma) \geq 0$ and $\sigma$ contains exactly one positive value.    
\end{mydef}

In \cite{s1949}, Sule\u{\i}manova stated that every such spectrum is realizable (for a proof via companion matrices, and references to other proofs, see Friedland \cite{f1978}). Fiedler \cite{f1974} proved that every Sule\u{\i}manova spectrum is symmetrically realizable. In \cite{jmp2015}, Johnson et al.~posed the following.

\begin{prob}
\label{jmpsul}
If $\sigma$ is a normalized Sule\u{\i}manova spectrum, is $\sigma$ realizable by a doubly stochastic matrix? 
\end{prob}
We will show that for Hadamard orders, the answer is `yes' and the realizing matrix can be taken to be symmetric.

%-------------------------------------------------------------------------------------------------------------------------------------------------------------------------------------
\begin{thm}
\label{thm:hadsul}
If $H$ is a normalized Hadamard matrix of order $n$ and $\sigma = \{ \lambda_1, \dots, \lambda_n \}$ is a normalized Sule\u{\i}manova spectrum, then $\sigma$ is realizable by a symmetric, doubly stochastic matrix. 
\end{thm}

\begin{proof}
It suffices to show that $v := [1~\cdots~\lambda_n]^\top \in \tope{H}$. For $k \in \langle n \rangle$, $k \neq 1$, let $D_k := e_1 e_1^\top - e_k e_k^\top$ and $M_k := \inv{n} H D_k H^\top$. Then
\begin{align*}
n M_k 
&= H  \left( e_1 e_1^\top - e_k e_k^\top \right) H^\top 								\\
&= H  e_1 \left( H e_1 \right)^\top - H e_k \left( H e_k \right)^\top  = J - H e_k \left( H e_k \right)^\top. 
\end{align*}
Because the matrix $H e_k \left( H e_k \right)^\top$ has entries in $\{\pm 1\}$, the matrix $J - H e_k \left( H e_k \right)^\top$ has entries in $\{0,2\}$, i.e., $M_k \geq 0$. Moreover, $\Diag{D_k} = e_1 - e_k$ so that $e_1 - e_k \in \tope{H}$. 

Since $\tope{H}$ is convex and $\{e_1,e_1-e_2,\dots,e_1-e_n\} \subset \tope{H}$, it follows that $\conv{e_1, e_1 - e_2, \dots, e_1 - e_n} \subseteq \tope{H}$. If $\mu_1 := s_1(\sigma)$ and $\mu_k := -\lambda_k$, for $k \geq 2$, then $\mu_k \geq 0$, $\sum_{k=1}^n \mu_k = 1$, and
\begin{align*}
v 
&= e_1 + \sum_{k = 2}^n \lambda_k e_k 								\\				
&= e_1 + \sum_{k = 2}^n (\lambda_k + \mu_k) e_1 + \sum_{k = 2}^n \lambda_k e_k 	\\
&= \mu_1 e_1 + \sum_{k = 2}^n \mu_i (e_1 - e_k). 
\end{align*}
Thus, $v \in \conv{e_1, e_1 - e_2, \dots,e_1 - e_n}$ and the result is established.
\end{proof}

%----------------------------------------------------------------------------------------------------------------
\begin{rem}
For any normalized Hadamard matrix, notice that 
\begin{equation*}
\mathcal{H}_n := \conv{e, e_1 - e_2, \dots, e_1 - e_n} \subseteq \tope{H}.
\end{equation*}
Since 
\begin{equation*}
\frac{1}{n} \left[ e + \sum_{k=2}^n (e_1 - e_k) \right] = e_1,
\end{equation*} it follows that 
\begin{equation*}
\conv{e_1, e_1 - e_2, \dots, e_1 - e_n} \subset \mathcal{H}_n \subseteq \tope{H}.
\end{equation*}
Thus, 
\begin{equation*}
\bigcap_{\mathscr{H}_n} \tope{H} = \mathcal{H}_n,
\end{equation*}
where $\mathscr{H}_n$ contains every normalized Hadamard matrix of order $n$.

In terms of the projected Perron spectratope, notice that
\begin{equation*}
\pi_1(\mathcal{H}_n) = \conv{e,-e_2, \dots, -e_n} \subseteq \mathcal{P}^1(H) \subset \mathcal{T}^{n-1}
\end{equation*}
If 
\begin{equation*}
S=
\begin{bmatrix}
1 & e^\top  \\
e & -I 
\end{bmatrix} \in \gl{n}{\bb{R}},~n\geq 2
\end{equation*}
then $\det{S} = (-1)^{n+1} n$. Following \eqref{simplexvol},
\begin{equation*}
\vol{\mathcal{P}^1(H)} \geq \vol{\pi_1(\mathcal{H}_n)} = \frac{n}{(n-1)!} = \frac{n^2}{n!}. 
\end{equation*}
\end{rem}

The following result is corollary to \hyp{Corollary}{walshspec} and \hyp{Theorem}{thm:hadsul}.

%-------------------------------------------------------------------------------------------------------------------------------------------------------------------------------------
\begin{cor}
\label{cor:hadsulei}
If $\sigma = \{\lambda_1, \dots, \lambda_{2^n} \}$ is a normalized Sule\u{\i}manova spectrum, then $\sigma$ is realizable by a trisymmetric doubly stochastic matrix.
\end{cor}

In \cite{f1974}, Fiedler showed that every Sule\u{\i}manova spectrum is symmetrically realizable. However, his proof is by induction and therefore does not explicitly yield a realizing matrix (the computation of which is of interest for numerical purposes), which is common for many NIEP results. Indeed, according to Chu: 
\begin{adjustwidth}{2.5em}{0pt}
Very few of these theoretical results are ready for implementation to actually compute [the realizing] matrix. The most constructive result we have seen is the sufficient condition studied by Soules \cite{s1983}. But the condition there is still limited because the construction depends on the specification of the Perron vector -- in particular, the components of the Perron eigenvector need to satisfy certain inequalities in order for the construction to work. \cite[p.~18]{c1998}. 
\end{adjustwidth}
Thus, \hyp{Theorem}{thm:hadsul} and \hyp{Corollary}{cor:hadsulei} are constructive versions of Fiedler's result for Hadamard powers.

%----------------------------------
\begin{cor}
\label{boyhansul}
If $\sigma = \{\lambda_1,\dots,\lambda_n\}$ is a normalized Sule\u{\i}manova spectrum, then there is a nonnegative integer $N$ such that 
\begin{equation*}
\hat{\sigma} := \sigma\cup\{\overbrace{0,\dots,0}^N\}
\end{equation*}
is realized by a symmetric, doubly stochastic $(n+N)$-by-$(n+N)$ matrix. If $n+N$ is a power of two, then the realizing matrix is trisymmetric. 
\end{cor}

%----------------------------------
\begin{rem}
If the Hadamard conjecture holds, then $N \leq 3$. Moreover,  \hyp{Corollary}{boyhansul} is a constructive version of the celebrated Boyle-Handelman theorem \cite[Theorem 5.1]{bh1991} for Sule\u{\i}manova spectra.
\end{rem}

A natural variation of \hyp{Problem}{jmpsul} is the following.

%----------------------------------
\begin{prob}
\label{jmpsulzero}
If $\sigma = \{\lambda_1,\dots,\lambda_n\}$, $n\geq 2$ is a normalized Sule\u{\i}manova spectrum such that $s_1(\sigma) = 0$, is $\sigma$ realizable by a doubly stochastic matrix? 
\end{prob}
Although the trace-zero assumption is rather restrictive, we demonstrate it is a nontrivial problem: If $A$ is a $3$-by-$3$ doubly stochastic matrix and $\trace{A} = 0$, then
\begin{equation*}
A=
\begin{bmatrix}
0 & a & 1-a \\
1-a & 0 & a \\
a & 1-a & 0
\end{bmatrix},~a \in [0,1],
\end{equation*}
and $\sig{A} = \{1,-1/2 \pm \sqrt{3}(a-1/2)\ii\}$, where $\ii := \sqrt{-1}$. Clearly, $\sigma$ is real if and only if $a=1/2$. Thus, $\{1,-1/2,-1/2\}$ is the only normalized Sule\u{\i}manova spectrum that is realizable by a $3$-by-$3$ trace-zero doubly stochastic matrix. 

We conclude with the observation that the Perron spectratopes of $H_1$ and $H_2$ resolves \hyp{Problem}{jmpsulzero} when $n=2$ and $n=4$, respectively (see \hyp{Figure}{rniepnequalstwo} and \hyp{Figure}{hadspectwo}).

%A straightforward proof by induction shows that if $A_n = 1/(n-1)(J-I)$, $n \geq 2$, then 
%\begin{equation*}
%\sig{A_n} = 
%\left\{
%1,\overbrace{\frac{-1}{n-1},\dots,\frac{-1}{n-1}}^{n-1} 
%\right\}.
%\end{equation*}

%-----------------------------------------------------------------------------------
% Bib
%----------------------------------------------------------------------------------

\bibliography{master}
\bibliographystyle{abbrv}

\end{document}